\documentclass[5p]{elsarticle}
\usepackage{makeidx,enumerate,graphicx,amsfonts,amssymb}
\usepackage{graphicx,epsfig}
\usepackage{fontenc}
\usepackage[utf8]{inputenc}
\usepackage{mathtools}

\newtheorem{theorem}{Theorem}
\newtheorem{lemma}[theorem]{Lemma}
\newtheorem{corollary}[theorem]{Corollary}
\newtheorem{observation}[theorem]{Observation}
\newproof{proof}{Proof}

\begin{document}

\title{New Facets of the QAP-Polytope}

\author[iitk]{Pawan Aurora}
\ead{paurora@cse.iitk.ac.in}
\author[iitk]{Shashank K Mehta}
\ead{skmehta@cse.iitk.ac.in}

\address[iitk]{Indian Institute of Technology, Kanpur - 208016, India}

\begin{abstract}
The Birkhoff polytope is defined to be the convex hull of permutation matrices, 
$P_{\sigma}\ \forall \sigma\in S_n$. We define a second-order permutation matrix 
$P^{[2]}_{\sigma}$ in $\mathbb{R}^{n^2\times n^2}$ corresponding to a permutation $\sigma$ as
$(P^{[2]}_{\sigma})_{ij,kl} = (P_{\sigma})_{ij}(P_{\sigma})_{kl}$. We call the convex hull 
of the second-order permutation matrices, the {\em second-order Birkhoff polytope} and
denote it by ${\cal B}^{[2]}$. It can be seen that ${\cal B}^{[2]}$ is isomorphic to the
QAP-polytope, the domain of optimization in {\em quadratic assignment problem}.
In this work we revisit the polyhedral combinatorics of the QAP-polytope viewing it as
${\cal B}^{[2]}$. Our main contribution is the identification of an exponentially 
large set of new facets of this polytope. Also we present a general inequality of which all the known 
facets of this polytope as well as the new ones, that we present in this paper, are special instances.
We also establish the existence of more facets which are yet to be identified.
\end{abstract}


\maketitle

\section{Introduction}
The Birkhoff polytope is defined to be the convex hull of the permutation matrices, 
$P_{\sigma}\ \forall \sigma\in S_n$. 
We define a second-order permutation matrix $P^{[2]}_{\sigma}$ corresponding to a permutation $\sigma$ as
$(P^{[2]}_{\sigma})_{ij,kl} = (P_{\sigma})_{ij}(P_{\sigma})_{kl}$. We call the convex hull 
of the second-order permutation matrices, the {\em second-order Birkhoff polytope} ${\cal B}^{[2]}$.

Clearly, the vertices of the second order Birkhoff polytope are vertices of the unit
cube, $\{0,1\}^{n^2\times n^2}$. Such polytopes are called zero-one polytopes.

Among various definitions of the Quadratic Assignment problem (QAP), see \cite{KaibelPhD97}, one is 
$min\{\sum_{ij,kl}(A_{ik}B_{jl}+D_{ij,kl})$ $Y_{ij,kl}| Y\in {\cal B}^{[2]}\}$ \cite{PovhR2009} where $A,B$ are 
input matrices and $D$ is a diagonal matrix. This may also be stated as $min\{\langle (A\otimes B + D), Y\rangle|$
$Y\in {\cal B}^{[2]}\}$. Thus QAP is an optimization problem over ${\cal B}^{[2]}$.
%, here we have introduced the notation ${\cal B}^{[2]}$. 
In the literature \cite{KaibelPhD97} this polytope is referred to as QAP-polytope.

${\cal B}^{[2]}$ is a zero-one polytope as is the Birkhoff polytope. But unlike the latter which has only
$n^2$ facets, ${\cal B}^{[2]}$ has exponentially many known facets \cite{JungerK97,KaibelPhD97} and 
exponentially many additional facets are identified in this paper. 
%We also prove that there are more facets which are yet to be discovered.

We will identify a generic inequality such that all the previously known facets and the new
facets discovered in this paper are special instances of this inequality. We will also show that
${\cal B}^{[2]}$ must have some facets which are not the instances of this inequality. Which implies that
more facets are yet to be discovered.

\section{A non-linear description of ${\cal B}^{[2]}$}
Consider the {\em completely positive} program CP given below. Note that the variable matrix is required to be 
completely positive (constraint \ref{1.1}), a stronger condition compared to positive semidefiniteness.
In CP $Y$ is a $(n^2+1)\times (n^2+1)$ matrix of variables with index set $(([n]\times [n])\cup \{w\})
\times (([n]\times [n])\cup \{w\})$. Since $Y$ is also positive semidefinite, there exist vectors $u_{ij}$ 
for all $ij$ and $\omega$ in $\mathbb{R}^{n^2+1}$ such that $Y_{ij,kl}=u_{ij}^T\cdot u_{kl}$ and 
$Y_{ij,w}=u_{ij}^T\cdot \omega$. We will refer to these vectors as the vector realization of $Y$. 
%Note that it exists even if $Y$ is a semidefinite matrix.

\vspace*{-2mm}
\begin{subequations}
\begin{alignat}{2}
    \text{CP:}\quad\text{max } \: & \sum_{i,j\in [n]}Y_{ij,ij}\ &,\ & \: \text{subject to}  \notag\\
                 \quad & Y\in {\cal C}^* \label{1.1}\\
                       & Y_{ij,ik}=0\ &,\ & 1\leq i,j,k\leq n,\ j\neq k \label{1.2}\\
                       & Y_{ji,ki}=0\ &,\ & 1\leq i,j,k\leq n,\ j\neq k \label{1.3}\\
                       & Y_{\omega,\omega}=1 \  \label{1.4}\\
                       & Y_{ij,\omega}=Y_{ij,ij}\ &,\ & 1\leq i,j\leq n \label{1.5}
  \end{alignat}
\end{subequations}

\subsection{United Vectors}

Let $\omega$ be any fixed unit vector in $\mathbb{R}^n$. Then for every unit vector $v$, we call $u=(\omega+v)/2$
a {\em united vector with respect to $\omega$}.

\begin{observation}\label{obs1} With respect to a fixed unit vector $\omega$,\\
(i) a vector $u$ is united if and only if $u\cdot \omega = u^2$,\\
(ii) if $u_1$ and $u_2$ are mutually orthogonal united vectors,
then $u_1+u_2$ is also a united vector.\\
(iii) let $u_1,\dots,u_k$ be a set of pairwise orthogonal united vectors.
This set is maximal (i.e., no new united vector
can be added to it while preserving pairwise orthogonality) if and only if $\omega$ belongs
to the subspace spanned by these vectors if and only if $\sum_i u_i = \omega$ if and only if
$\sum_i u_i^2 = 1$.
\end{observation}

Consider CP in the light of united vectors. 
Let $Y$ be a solution with $Y_{ij,kl}=u^T_{ij}.u_{kl}$ for all $ijkl$.
 and $Y_{ij,\omega} = u_{ij}^T.\omega$. 
This is a $(n^2+1)\times (n^2+1)$ matrix in which last row and the last column are same as the diagonal
because $u_{ij}^T.\omega = u_{ij}^2$. In our subsequent discussion we will treat it as an $n^2\times n^2$ 
matrix by dropping the last row and the last column.
Equations \ref{1.4} and \ref{1.5} imply that $u_{ij}$ are united vectors. Equations \ref{1.2} and \ref{1.3}
imply that $\{u_{i1},\dots,u_{in}\}$ are orthogonal sets and so are $\{u_{1i},\dots,u_{ni}\}$. 
From the Observation \ref{obs1} we know that each of these sets add up to a vector of length at most $1$. 
Hence the objective function can evaluate to at most $n$. Note that in arriving at the upper bound we
did not make use of the fact that $Y$ is completely positive. Hence the same bound also holds for the
positive semidefinite relaxation of CP (call it SDP).

Let ${\cal L}$ be the feasible region of CP where the objective function attains its maximum value, $n$.
Observe that $P^{[2]}_{\sigma}\in{\cal L}$ $\forall$ $\sigma\in S_n$, since these are completely positive
rank-1 matrices. 
Clearly ${\cal B}^{[2]}\subseteq{\cal L}$. In fact the converse is also true as the following lemma (proved in
\ref{CPFeasPrf}) shows. 

\begin{lemma}\label{cpfeas}
${\cal L}\subseteq{\cal B}^{[2]}$.
\end{lemma}

Consequence of the above observations is that ${\cal L}={\cal B}^{[2]}$. Hence we 
have a non-linear description of ${\cal B}^{[2]}$. A similar non-linear description 
appears in \cite{PovhR2009}. In the rest of this paper our objective is to develop a linear description for
${\cal B}^{[2]}$.

\section{The Affine Plane of ${\cal B}^{[2]}$}

In this section we will develop a system of equations whose solution is the affine plane of
${\cal B}^{[2]}$, i.e., the affine-hull of all $P^{[2]}_{\sigma}$s. 

\subsection{A System of Equations for the Affine Plane}

Consider a solution $Y$ of CP (or SDP). If the objective function achieves its maximum value for $Y$, then
each set $\{u_{i1},\dots,u_{in}\}$ is a maximal orthogonal set. Similarly each set $\{u_{1i},\dots,u_{ni}\}$ 
is also a maximal orthogonal set. In that case from united vector property $\sum_{i}u_{ij}=\sum_ju_{ij} = \omega$. 
We then have $1 = \omega^T\cdot\omega = \sum_i\omega^T\cdot u_{ij} = \sum_iu_{ij}^T\cdot u_{ij} = \sum_iY_{ij,ij}$.
Similarly $\sum_jY_{ij,ij}=1$. We also have $\sum_kY_{ij,kl} = u_{ij}^T\cdot (\sum_k u_{kl})
=u_{ij}^T\cdot\omega=Y_{ij,ij}$. Similarly $\sum_{l} Y_{ij,kl} = Y_{ij,ij}$. So we have the following linear conditions:
\vspace*{-5mm}

\begin{subequations}
\begin{align}
Y_{ij,kl}-Y_{kl,ij} &= 0 \,&\forall i,j,k,l \label{2.1}\\
Y_{ij,il}=Y_{ji,li} &= 0 \,&\forall i,\forall j\neq l \label{2.2}\\
\sum_{k} Y_{ij,kl}= \sum_{k} Y_{ij,lk}   &=Y_{ij,ij}  \,&\forall i,j,l \label{2.3}\\
\sum_{j}Y_{ij,ij}= \sum_{j}Y_{ji,ji} &= 1 \,&\forall i\label{2.4}
\end{align}
\end{subequations}

It is easy to verify that every $P^{[2]}_{\sigma}$ satisfies these equations. We make a stronger claim in the
following lemma (proved in \ref{apx_p2s-01}).
%The proof of the following lemmas are in \ref{apx_p2s-01} and \ref{affine-plane}.

\begin{lemma}\label{p2s-01}
The only $0/1$ solutions of Equations \ref{2.1}-\ref{2.4} are the $P^{[2]}_{\sigma}$s.
\end{lemma}

The following lemma (proved in \ref{affine-plane}) sums up the main result of this section.

\begin{lemma} \label{affine} The solution plane $P$ of equations \ref{2.1}-\ref{2.4} is the affine plane
spanned by $P^{[2]}_{\sigma}$s, i.e., $P=\{\sum_{\sigma}\alpha_{\sigma}P^{[2]}_{\sigma}|$ $\sum_{\sigma}
\alpha_{\sigma}=1 \}$.
\end{lemma}

\section{Some Facets of ${\cal B}^{[2]}$}\label{facet}

To develop a linear description of ${\cal B}^{[2]}$, we need the inequalities corresponding to its facets.
The complete linear description will be these inequalities along with equations \ref{2.1}-\ref{2.4}.
In this section we will identify exponentially many new facets of ${\cal B}^{[2]}$, in addition to 
exponentially many already known facets given in \cite{JungerK97,KaibelPhD97}.

We will represent a facet by an inequality $f(x)\geq 0$ which defines the half space that contains the polytope
and the plane $f(x)=0$ contains the facet.

Let $\{\omega\}\cup \{u_{ij}|1\leq i,j\leq n\}$ represent a (united) vector realization of any point 
$Y\in {\cal B}^{[2]}$. Define a vector $A=\sum_{ij}n_{ij}u_{ij}$
for some choice of $n_{ij}\in \mathbb{Z}$ and let $\beta \in \mathbb{Z}$. Consider the following inequality.

\vspace*{-2mm}
\begin{equation}\label{gen-ineq}
(A-(\beta-0.5) \omega)^2 \geq 0.25. 
\end{equation}

The above inequality defines the half space 
$\sum_{ij}(n_{ij}^2-(2\beta-1)n_{ij})$ $Y_{ij,ij} + \sum_{ij\neq kl}n_{ij}n_{kl} Y_{ij,kl} + \beta^2 - \beta \geq 0$.
The united vector realization of $P^{[2]}_{\sigma}$ is $u_{ij}=\omega$ if $\sigma(i)=j$, else $u_{ij}=0$.
It is easy to see that every $P^{[2]}_{\sigma}$, hence every point of ${\cal B}^{[2]}$, satisfies the inequality 
(\ref{gen-ineq}).

If there 
exists a permutation $\sigma$ such that $\sum_{(ij): \sigma(i)=j} n_{ij}$ is either equal to $\beta$ or $\beta-1$, 
then $P^{[2]}_{\sigma}$ satisfies (\ref{gen-ineq}) with equality. In this case the plane 
$(A-(\beta-0.5) \omega)^2 = 0.25$
is a supporting plane of ${\cal B}^{[2]}$ and hence defines a face. We will show that some instances of this 
inequality define facets for ${\cal B}^{[2]}$. Later we will 
also show that all the facets identified in \cite{JungerK97,KaibelPhD97} also belong to the same inequality.

It may be pointed out that another inequality, which can define faces, is $(A-\beta\omega)^2\geq 0$. But no known
facets correspond to this inequality.

We will discuss the following three sets of inequalities:\\
1. $(u_{ij}+u_{kl} -0.5\omega)^2\geq 0.25$,\\
2. $(u_{p_1q_1}+u_{p_2q_2}+u_{p_1q_2}-u_{kl}-0.5\omega)^2\geq 0.25$,\\
3. $(u_{i_1j_1}+\dots + u_{i_mj_m}-u_{kl}-0.5\omega)^2\geq 0.25$,
and show that each instance of each of these inequalities defines a facet of ${\cal B}^{[2]}$.

The following lemma gives a method to establish a facet.

Let $X$ be a set of vectors. Then $LS(X)$ denotes the subspace spanned by the vectors of $X$.

\begin{lemma}\label{st} Let $V$ be the set of vertices of a polytope such that the affine plane
of $V$ does not contain the origin and $f(x)\geq 0$ be a linear inequality satisfied by all the 
vertices. Let $S=\{v\in V|f(v)=0\}$ such that $V\setminus S\neq\emptyset$. And let $v_0\in V\setminus S$ 
such that each vertex in $V$ can be expressed as a linear combination of $\{v_0\}\cup S$. Then $S$ is 
a facet, i.e., $f(x)\geq 0$ defines a facet. 
\end{lemma}

\proof{Let $d$ denote the dimension of $LS(V)$. So the dimension of the affine plane of $V$ is $d-1$.
Also $V\subset LS(\{v_0\}\cup S)$ so the dimension of $LS(S)$ is at least $d-1$. As the affine plane
of $S$ does not contain the origin, the dimension of the affine plane of $S$ is at least $d-2$. 
Observe that $V$ is not contained in $LS(S)$ since $f(x)$ is non-zero for $x\in V\setminus S$.
We conclude that the dimension of the affine plane of $S$ is exactly one less than that of the 
affine plane of $V$.
$\Box$}

\begin{corollary}\label{stc} Let $G=(V\setminus S,E)$ be a graph with the property that for each $\{u,v\}\in E$,
$u-v \in LS(S)$. If $G$ is connected, then $S$ is a facet.
\end{corollary}

Let $k_1,k_2,k_3$ be any three integers belonging to $[n]$. Let $\sigma_1,\dots,\sigma_6$ be a set of permutations 
of $S_n$ which have same image for each element of $[n]\setminus \{k_1,k_2,k_3\}$, i.e.,
$\sigma_i(z)=\sigma_j(z)$ for all $z\in [n]\setminus \{k_1,k_2,k_3\}$ for every $i,j\in \{1,\dots,6\}$.
Let images of $k_1,k_2,k_3$ under $\sigma_1,\dots,\sigma_6$ be $(a,b,c),$ $(a,c,b),$ $(b,a,c),$ $(b,c,a),$ $(c,a,b),$
$(c,b,a)$ respectively. Further, suppose $x,y$ be any two elements of $[n]\setminus \{k_1,k_2,k_3\}$. 
Let $\sigma'_i$ be transposition of $\sigma_i$ on indices $x$ and $y$, for each $i=1,\dots,6$.
That is, $\sigma'_i(z)=\sigma_i(z)$ for all $z\in [n]\setminus \{x,y\}$, $\sigma'_i(x)=\sigma_i(y)$,
and $\sigma'_i(y)=\sigma_i(x)$. Let $\Sigma=\{\sigma_1,\dots,\sigma_6,$ $\sigma'_1,\dots,\sigma'_6\}$.
Following is a useful identity.

\begin{lemma}\label{zero} Let $\sigma_1,\dots,\sigma_6,\sigma'_1,\dots,\sigma'_6$ be a set of permutations
as defined above. Then $\sum_{\sigma\in\Sigma}sign(\sigma)P^{[2]}_{\sigma}=0$.
%$P^{[2]}_{\sigma_1}-P^{[2]}_{\sigma_2}-P^{[2]}_{\sigma_3}+P^{[2]}_{\sigma_4}
%+P^{[2]}_{\sigma_5}-P^{[2]}_{\sigma_6}  
%-P^{[2]}_{\sigma'_1}+P^{[2]}_{\sigma'_2}+P^{[2]}_{\sigma'_3}-P^{[2]}_{\sigma'_4}
%-P^{[2]}_{\sigma'_5}+P^{[2]}_{\sigma'_6}=0$.
\end{lemma}

In this section $V$ will denote the set $\{P^{[2]}_{\sigma}|\sigma\in S_n\}$ and $S$ will denote $\{P^{[2]}_{\sigma}|
f(P^{[2]}_{\sigma})=0\}$.

\begin{theorem}\label{nonneg}
The non-negativity constraint $Y_{ij,kl}\geq 0$, wh\-i\-ch is same as $(u_{ij}+u_{kl} -0.5\omega)^2\geq 0.25$,
defines a facet of ${\cal B}^{[2]}$ for every $i,j,k,l$ such that $i\neq k$ and $j\neq l$.
\end{theorem}
\proof{
Observe that the non-negativity condition is satisfied by every $P^{[2]}_{\sigma}$. Every vertex in
the set $V\setminus S$ corresponds to a permutation $\sigma$ where $\sigma(i)=j$ and $\sigma(k)=l$.
Consider a graph $G=(V\setminus S,E)$ where $E={\{P^{[2]}_{\sigma},P^{[2]}_{\sigma'}\}}$ where 
$\sigma'\text{ is a transposition of }\sigma$.
Since the set of permutations corresponding to the vertices in $V\setminus S$ is isomorphic to the group 
$S_{n-2}$, $G$ must be a connected graph.

Let $\sigma_1$ and $\sigma_1'$ be a pair of permutations in $V\setminus S$ which are transpositions of each 
other at indices $x,y$, i.e., $\sigma_1(x)=\sigma'_1(y)$ and $\sigma_1(y)=\sigma'_1(x)$. 
Let $k_1=i,k_2=k$ and $k_3$ be any element
other than $x,y,i,k$. Consider all the permutations $\sigma_2,\dots,\sigma_6,\sigma_2',\dots,\sigma_6'$
as defined in the context of lemma \ref{zero}. Observe that all the $P^{[2]}_{\sigma}$s corresponding to these ten
permutations belong to $S$. Hence we can express $P^{[2]}_{\sigma_1}-P^{[2]}_{\sigma_1'}$ in terms of vertices
in $S$ using the identity of the lemma. From corollary \ref{stc} the inequality defines a facet. $\Box$ }

\begin{lemma}\label{connected} (1) Let $X$ be a set of permutations $\sigma$ such that $\sigma(1)=1,\dots,
\sigma(a)=a$ and $\sigma(a+1)\notin I_1,\dots,\sigma(a+b)\notin I_b$ where 
$I_j$ are subsets of $[n]$ such that $|\cup_iI_i|\leq n-a-b$. Let 
$G=(X,E)$ be a graph in which $\{P^{[2]}_{\sigma},P^{[2]}_{\sigma'}\}\in E$ if $\sigma$ and $\sigma'$ are 
transpositions of each other. Then $G$ is connected.

(2) Let $X$ be a set of permutations $\sigma$ such that $\sigma(1) =1,\dots,\sigma(a)=a, \sigma(a+1)\neq x_1,
\dots,\sigma(a+b)\neq x_b$, where all $x_i$ are distinct and greater than $a$ and $a+b<n$. Let
$G=(X,E)$ be a graph in which $\{P^{[2]}_{\sigma},P^{[2]}_{\sigma'}\}\in E$ if $\sigma$ and $\sigma'$ are 
transpositions of each other. Then $G$ is connected.
\end{lemma}

\proof{ (1) Let $I=\cup_iI_i$. Without loss of generality assume that $I=\{a+b+1,a+b+2,
\dots\}$. Hence the identity permutation belongs to $X$. 

Given any permutation $\sigma_0\in X$, we will show that there is a path from $\sigma_0$ to the identity.
Starting from $\sigma_0$, suppose the current permutation $\sigma$ is such that for some $i\in \{a+1,\dots,
a+b\}$, $\sigma(i)\in \{a+b+1,\dots,n\}$. Hence there must exist a $j\in \{a+b+1,a+b+2,\dots, n\}$ such that
$\sigma(j)\in \{a+1,\dots,a+b\}$. 
Performing transposition on $i$ and $j$ we extend the path as the new permutation is also in $X$. Finally 
we will reach a permutation in which all indices in the range $a+1,\dots,a+b$ map to $a+1,\dots,a+b$ and 
hence all indices of $a+b+1,\dots n$ map to $a+b+1,\dots,n$.

Next perform transpositions within indices of $a+1,$ $\dots,$ $a+b$ so that finally $\sigma(i)$ maps to $i$ 
for all $i$ in this range. Note that all the permutations generated in the process belong to $X$. In the 
end we do the same for indices in the range $a+b+1,\dots,n$. 

(2) The claim is vacuously true if $X$ is empty. So we assume that it is non-empty. 
By relabeling we can make sure that $x_i\neq a+i$ for all $1\leq i \leq b$.
So without loss of generality we can assume that the identity permutation belongs to $X$. To prove 
the claim we will show that starting from
any arbitrary permutation $\sigma_0\in X$ there is a path from $\sigma_0$ to the identity permutation.
While tracing this path, the current permutation has $\sigma(a+i)\neq a+i$ while $\sigma(a+j)=a+j$
for all $j<i$. Let $\sigma^{-1}(a+i)=a+k$. 

If $a+k>a+b$ or if $a+k\leq a+b$ \& $\sigma(a+i)\neq x_k$, then perform transposition on indices $a+i$ and $a+k$
resulting into the new permutation $\sigma'$ which belongs to $X$ and is "closer" to the identity.

Now consider the case where $\sigma(a+i)= x_k$. Observe that there must be at least three indices beyond
$a+i-1$. Let $a+j$ be any index greater than $a+b$. Perform transposition on indices $a+j$ and $a+k$
giving $\sigma'$ and then perform transposition on $a+i$ and $a+j$. Let the new permutation be $\sigma''$.
Observe that both, $\sigma'$ and $\sigma''$, belong to $X$. So the path extends by edges $\{\sigma,\sigma'\}$
and $\{\sigma',\sigma''\}$. Further, $\sigma''$ is closer to the identity.

Thus the path eventually reaches the identity and its length is at most $2b$ steps. $\Box$}

\begin{theorem}\label{secondset}
Inequality $Y_{p_1q_1,kl}+Y_{p_2q_2,kl}+Y_{p_1q_2,kl}\leq Y_{kl,kl}+Y_{p_1q_1,p_2q_2}$,
which is same as $(u_{p_1q_1}+u_{p_2q_2}+u_{p_1q_2}-u_{kl}-0.5\omega)^2\geq 0.25$, defines a facet of 
${\cal B}^{[2]}$, where $p_1,p_2,k$ are distinct and $q_1,q_2,l$ are also distinct and $n\geq 6.$
\end{theorem}

\proof{The set of vertices which satisfy the inequality strictly is the union of
$X_1=\{P^{[2]}_{\sigma}|\sigma(p_1)=q_1,\sigma(p_2)=q_2,\sigma(k)\neq l\}$ and
$X_2=\{P^{[2]}_{\sigma}|\sigma(p_1)\neq q_1,$ $\sigma(p_1)\neq q_2,$ $\sigma(p_2)\neq q_2,$
$\sigma(k)=l\}$. So $V\setminus S = X_1\cup X_2$.

Define a graph $G=(X_1\cup X_2,E)$ where $E$ is the set of edges $\{P^{[2]}_{\sigma},P^{[2]}_{\sigma'}\}$ where $\sigma 
\text{ is a transposition of }\sigma'$ and both matrices belong to $X_1\cup X_2$.
 From lemma \ref{connected} the subgraphs on $X_1$ and $X_2$ are each connected.
We also notice that there is no edge connecting these components. So we add a special edge
$\{P^{[2]}_{\alpha_1},P^{[2]}_{\alpha_2}\}$
to $G$ making the graph connected. Let $\alpha_1$ be any arbitrary member of $X_1$. Let
$i_2=\alpha_1^{-1}(l)$ and $r$ be any index other than $p_1,p_2,k,i_2$. So $\alpha_1$ maps
$p_1 \rightarrow q_1, p_2\rightarrow q_2, k\rightarrow b, i_2\rightarrow l,r\rightarrow a$ for some $a$ and $b$.
Define $\alpha_2$ to be the permutation which maps
$p_1 \rightarrow a, p_2\rightarrow q_1, k\rightarrow l, i_2\rightarrow b,r\rightarrow q_2$ and in all
other cases images of $\alpha_1$ and $\alpha_2$ coincide. Observe that $P^{[2]}_{\alpha_2}\in X_2$.

Now we will show that for each edge $\{P^{[2]}_{x'},P^{[2]}_{y'}\}$ of the graph,
$P^{[2]}_{x'}-P^{[2]}_{y'}$ belongs to $LS(S)$. We begin with the edge
$\{P^{[2]}_{\alpha_1},P^{[2]}_{\alpha_2}\}$. Let $\sigma_1=\alpha_1$. Define $\sigma_2,\dots,\sigma_6$ using
$k_1=p_1,k_2=p_2,
k_3=r$ as described before lemma \ref{zero}. Taking $x=k$ and $y=i_2$, define $\sigma'_1, \dots, \sigma'_6$.
See that $\alpha_2 = \sigma'_5$. The rest of the permutations are in $S$.
Hence from lemma \ref{zero} $P^{[2]}_{\alpha_1} -
P^{[2]}_{\alpha_2}$ can be expressed as a linear combination of vertices in $S$.

Next we will show that each edge in the graph on $X_1$ has the same property.
Let $\{P^{[2]}_{\sigma_1},P^{[2]}_{\sigma'_1}\}$ be an edge in the graph on $X_1$. In both permutations $p_1$ and $p_2$ map to $q_1$ and
$q_2$ respectively. Define $k_1=p_1$ and $k_2=p_2$. There is at least one index, other than $p_1,p_2,k$, which
has the same image in both the permutations because $n\geq 6$. Label it $k_3$. Consider $5$ new permutations formed
from $\sigma_1$ by permuting the images of $k_1,k_2$ and $k_3$. Call them $\sigma_2,\dots,\sigma_6$. Similarly
define $\sigma'_2,\dots,\sigma'_6$ from $\sigma'_1$. Observe that in each $\sigma_i$ for $i\geq 2$, $k$ does not
map to $l$. In addition either $p_1$ does not map to $q_1$ or $p_2$ does not map to $q_2$. Hence
$P^{[2]}_{\sigma_2},\dots, P^{[2]}_{\sigma_6}$ belong to $S$. Similarly
$P^{[2]}_{\sigma'_2},\dots,P^{[2]}_{\sigma'_6}$ also belong to $S$. From lemma \ref{zero},
$P^{[2]}_{\sigma_1}-P^{[2]}_{\sigma'_1}\in LS(S)$.

Now we consider the edges of $X_2$. Let $\{P^{[2]}_{\sigma_1},P^{[2]}_{\sigma'_1}\}$ be one such edge. Let
$x,y$ be the indices at which $\sigma_1$ and $\sigma'_1$ differ. Consider two cases of $\sigma_1$:
(1) $\sigma_1(p_1)=a, \sigma_1(p_2)=b, \sigma_1(k)= l, \sigma_1(r)=q_1, \sigma_1(s)= q_2$,
(2) $\sigma_1(p_1)= a, \sigma_1(p_2)= q_1, \sigma_1(k)=l, \sigma_1(r)= q_2$.

Case (1) Subcase $|\{p_1,p_2,r,s\}\cap \{x,y\}|\leq 1$:
If $p_1\notin \{x,y\}$, then define $k_1=k$, $k_2=p_1$, and $k_3$ be any index in $\{r,s\}\setminus \{x,y\}$.
Otherwise $k_1=k, k_2=p_2,k_3=s$. All the permutations $\sigma_2,\dots,\sigma_6$ and $\sigma'_2,\dots,\sigma'_6$
as defined before lemma \ref{zero} are in $S$. So $P^{[2]}_{\sigma_1}- P^{[2]}_{\sigma'_1}$ can be expressed as
a linear combination of points in $S$ using the identity.

Subcase $\{x,y\}\subset \{p_1,p_2,r,s\}$: Only three cases are possible here: $x=p_2,y=r$; $x=r,y=s$;
and $x=p_1,y=p_2$, apart from exchanging the roles of $x$ and $y$. In the first case
let $k_1=p_1, k_2=s, k_3=k$ and use lemma \ref{zero}.
The remaining two cases are proven differently.

In these two cases we will not show that $P^{[2]}_{\sigma}-P^{[2]}_{\sigma'}$ can be expressed
as a linear combination of vertices in $S$. Instead, we will delete such edges from $E$ and show that
the reduced graph is still connected. Consider an edge $\{P^{[2]}_{\sigma},P^{[2]}_{\sigma'}\}$ of the second type where
$\sigma$ maps: $p_1\rightarrow a, p_2 \rightarrow b, k\rightarrow l, r\rightarrow q_1, s\rightarrow q_2, 
u\rightarrow v$ and $\sigma'$ maps: $p_1\rightarrow a, p_2 \rightarrow b, k\rightarrow l, r\rightarrow q_2, 
s\rightarrow q_1, u\rightarrow v$. Note that $n\geq 6$ so there always exists the pair $u,v$.
Rest of the indices have the same images in the two permutations. To show that after dropping an
edge of this class the graph remains connected, define two new permutations: $\alpha_1$:
$p_1\rightarrow a, p_2 \rightarrow b, k\rightarrow l, r\rightarrow v, s\rightarrow q_2, u\rightarrow q_1$
and $\alpha_2$: $p_1\rightarrow a, p_2 \rightarrow b, k\rightarrow l, r\rightarrow q_2, s\rightarrow v, 
u\rightarrow q_1$. Other mappings are same as in $\sigma$. Observe that
$\{P^{[2]}_{\sigma},P^{[2]}_{\alpha_1}\}, \{P^{[2]}_{\alpha_1},P^{[2]}_{\alpha_2}\}$ and $\{P^{[2]}_{\alpha_2},
P^{[2]}_{\sigma'}\}$ are edges in the reduced graph, hence there is a path from $P^{[2]}_{\sigma}$ to $P^{[2]}_{\sigma'}$
in it.

Let $\{P^{[2]}_{\sigma},P^{[2]}_{\sigma'}\}$ be third type of edge. So $\sigma$ maps
$p_1 \rightarrow a, p_2 \rightarrow b, k\rightarrow l, r\rightarrow q_1, s\rightarrow q_2$ and $\sigma'$ maps
$p_1\rightarrow b, p_2 \rightarrow a, k\rightarrow l, r\rightarrow q_1, s\rightarrow q_2$. Again to show
 a path from $P^{[2]}_{\sigma}$ to $P^{[2]}_{\sigma'}$ in the reduced graph, define $\alpha_1$:
$p_1\rightarrow a, p_2 \rightarrow q_1, k\rightarrow l, r\rightarrow b, s\rightarrow q_2$ and $\alpha_2$:
$p_1\rightarrow b, p_2 \rightarrow q_1, k\rightarrow l, r\rightarrow a, s\rightarrow q_2$. Other mappings
are same as in $\sigma$. In case (2) we will show that $\{P^{[2]}_{\alpha_1},P^{[2]}_{\alpha_2}\}$
and $\{P^{[2]}_{\alpha_2},P^{[2]}_{\sigma'}\}$ are edges in the reduced graph.
Hence $P^{[2]}_{\sigma},P^{[2]}_{\alpha_1},P^{[2]}_{\alpha_2},P^{[2]}_{\sigma'}$ is a path in the reduced graph.
Note that $\{P^{[2]}_{\sigma},P^{[2]}_{\alpha_1}\}$ is an edge of the first type.

Case (2) Subcase $\{p_1,p_2,r=\sigma^{-1}(q_2)\}\cap \{x,y\}=\emptyset$: In this case define $k_1=p_1,k_2=p_2,k_3=r$.

Subcase $|\{p_1,p_2,r=\sigma^{-1}(q_2)\}\cap \{x,y\}|=1$: If $x=p_1$ or $y=p_1$, then $k_1=p_2, k_2=r,k_3=k$.
If $x=p_2$ or $y=p_2$, then $k_1=p_1,k_2=r,k_3=k$. Finally if $x=r$ or $y=r$, then
$k_1=p_1,k_2=p_2,k_3=k$. In each case lemma \ref{zero} gives a desired linear expression in terms
of points in $S$ for $P^{[2]}_{\sigma_1}-P^{[2]}_{\sigma'_1}$.

Subcase $\{x,y\}\subset \{p_1,p_2,r=\sigma^{-1}(q_2)\}$ does not arise because in this case every transposition leads
to a permutation in $S$.

From Corollary \ref{stc} we conclude that $S$ is a facet. $\Box$
}

Consider the following inequality where $3\leq m\leq n-3$.
\vspace*{-2mm}
\begin{equation}\label{ineqsys}
Y_{i_1j_1,kl}+Y_{i_2j_2,kl}+\ldots+Y_{i_mj_m,kl}\leq Y_{kl,kl}+\sum_{r\neq s}Y_{i_rj_r,i_sj_s}.
\end{equation}
Observe that it can also be written as $(u_{i_1j_1}+\dots + u_{i_mj_m}-u_{kl}-0.5\omega)^2\geq 0.25$, 
where $3\leq m\leq n-3$.  
In the rest of this section we will show that inequality (\ref{ineqsys}) 
also defines a facet of $\mathcal{B}^{[2]}$.

We will continue to use $S$ to denote the set of vertices that satisfy the given inequality with equality. Let $T$ 
denote the set of remaining vertices. For the system (\ref{ineqsys}) the set $T$ can be subdivided into
the following classes:

\begin{enumerate}
\item $T_1: k\rightarrow l,i_1\not\rightarrow j_1,i_2\not\rightarrow j_2,\ldots,i_m\not\rightarrow j_m$.
\item $T_2: k\rightarrow l$ and three or more $i_r\rightarrow j_r$.
\item $T_3: k\not\rightarrow l$ and two or more $i_r\rightarrow j_r$.
\end{enumerate}

In classes $T_2$ and $T_3$ we do further subdivision. If a permutation in $T_2$ maps $i_r$ to $j_r$ for $x$
out of $m$ indices, then such a permutation belongs to subclass denoted by $T_{2,x}$. Similarly $T_{3,x}$
is defined. Observe that $T_2=\cup_{x\geq 3}T_{2,x}$ and $T_3=\cup_{x\geq 2}T_{3,x}$.

\begin{lemma}
Let $m\geq 3$. The graph $G_1$ on $T_1$, with edge set $\{P^{[2]}_{\sigma'},P^{[2]}_{\sigma''}\}$ where
$\sigma'$ is a transposition of $\sigma''$, is connected.
Further the difference vector corresponding to each edge belongs to $LS(S)$.
\end{lemma}

\proof{The first part of the lemma is a special case of the second part of lemma \ref{connected}.

For the second part let $\{P^{[2]}_{\sigma_1},P^{[2]}_{\sigma'_1}\}$ be an edge in $G_1$ where $\sigma_1(x)=\sigma'_1(y)$ and 
$\sigma_1(y)=\sigma'_1(x)$. As $m$ is at least $3$, there exists $r\leq m$ such that $i_r\notin \{x,y\}$ 
and $j_r\notin \{\sigma_1(x), \sigma_1(y)\}$. Without loss of generality assume that $r=1$. So we have 
description of $\sigma_1$ and $\sigma'_1$ as follows:
$\sigma_1: k\rightarrow l, x\rightarrow \alpha, y\rightarrow \beta, i_1\rightarrow \gamma, \delta 
\rightarrow j_1, \dots$ and $\sigma'_1: k\rightarrow l, x\rightarrow \beta, y\rightarrow \alpha, i_1
\rightarrow \gamma, \delta \rightarrow j_1, \dots$. 

Taking $k_1=k, k_2=i_1, k_3=\delta$, $x$ as $x$ and $y$ as $y$, generate permutations $\sigma_2,\dots,
\sigma_6,\sigma'_2, \dots,\sigma'_6$ as defined before lemma \ref{zero}. Vertices corresponding to each 
of these permutations belong to $S$. Hence from lemma \ref{zero}, $P^{[2]}_{\sigma_1}-P^{[2]}_{\sigma'_1}
\in LS(S)$. $\Box$
}

\begin{corollary}\label{T1}
Given any $P^{[2]}_{\sigma*}$ in $T_1$, each $P^{[2]}_{\sigma}$ in $T_1$ belongs to $LS(\{P^{[2]}_{\sigma*}\}\cup S)$. 
\end{corollary}

\begin{lemma}\label{3-2}
Let $n\geq 5$. Then  $T_{3,2}\subset LS(T_1\cup S)$.
\end{lemma}

\proof{
Consider any arbitrary permutation, $\sigma$, with the corresponding vertex belonging to $T_{3,2}$.
Let $\beta=\sigma^{-1}(l)$ and $\gamma$ be any arbitrary element from $[n]\setminus \{k,i_1,i_2,\beta\}$. The 
description of $\sigma$ is: $k \rightarrow \alpha, i_1\rightarrow j_1, i_2\rightarrow j_2, \beta 
\rightarrow l, \gamma \rightarrow \delta$ and all other maps are different from $(i_p,j_p)$ for any $p$,
where $\alpha\neq l$. Our goal is to show that $P^{[2]}_{\sigma}\in LS(T_1\cup S)$. Consider two cases.

Case: $(\beta, \alpha) \neq (i_p,j_p)$ for any $p$. Take $\sigma_1=\sigma, k_1=i_1, k_2=i_2, k_3=\gamma, x=k, 
y=\beta$. All the vertices corresponding to permutations $\sigma_2,\dots,\sigma_6,\sigma'_1,\dots,\sigma'_6$ 
generated with these parameters belong to $S\cup T_1$. From lemma \ref{zero} $P^{[2]}_{\sigma_1}\in LS(T_1\cup S)$.

Case: $(\beta,\alpha)=(i_3,j_3)$. In this case $\sigma: k \rightarrow j_3, i_1\rightarrow j_1, i_2\rightarrow j_2, 
i_3 \rightarrow l, \gamma \rightarrow \delta$. Take $\sigma_1=\sigma,k_1=k, k_2=i_1, k_3=i_2, x=i_3, y=\gamma$. 
Then we see that $P^{[2]}_{\sigma_1}$ and $P^{[2]}_{\sigma'_1}$ both belong to $T_{3,2}$ and the vertices corresponding 
to the remaining ten permutations belong to $S$. So $P^{[2]}_{\sigma_1}-P^{[2]}_{\sigma'_1} \in LS(S)$. 
Now from the first case $P^{[2]}_{\sigma'_1}$ belongs to $LS(T_1\cup S)$. Therefore $P^{[2]}_{\sigma_1}$ also 
belongs to $LS(T_1\cup S)$.  $\Box$
} 

\begin{lemma}\label{2-3}
Let $n\geq 5$. Then $T_{2,3}\subset LS(T_1\cup S)$.
\end{lemma}

\proof{
Let $P^{[2]}_{\sigma}$ be an arbitrary element of $T_{2,3}$.
We will express $P^{[2]}_{\sigma}$ as a linear combination of some members of $T_{3,2}\cup S$. The rest 
will follow from lemma \ref{3-2}.

Without loss of generality assume that the given permutation $\sigma$ in $T_{2,3}$ maps $k\rightarrow l,
i_1\rightarrow j_1,i_2\rightarrow j_2, i_3\rightarrow j_3$. Also let $\sigma$ map $\alpha\rightarrow \beta$ 
for some $\alpha\notin \{k,i_1,i_2,i_3\}$. Now generate the permutations $\sigma_2,\dots,\sigma_6,\sigma'_1,
\dots,\sigma'_6$ with parameters $\sigma_1=\sigma,k_1=i_2,k_2=i_3,k_3=\alpha,x=k,y=i_1$. See that 
$P^{[2]}_{\sigma'_1}\in T_{3,2}$ 
and the remaining ten permutations belongs to $S$. So $P^{[2]}_{\sigma_1}-
P^{[2]}_{\sigma'_1} \in LS(S)$. From lemma \ref{3-2}, $P^{[2]}_{\sigma}\in LS(S\cup T_1)$. 
$\Box$}

\begin{lemma}
Let $n\geq 6$. Given any $P^{[2]}_{\sigma}$ in $T_{2,r}$ (resp. $T_{3,r}$) with $r>3$ (resp. $r>2$), it can be  
expressed as a linear combination of elements of $T_1\cup S$.
\end{lemma}

\begin{proof}
Let $P^{[2]}_{\sigma_1}\in T_{3,r}$ with $r\geq 3$. Assume that $\sigma_1$ maps $\alpha\rightarrow l, k\rightarrow
\gamma, i_1\rightarrow j_1, i_2 \rightarrow j_2, i_3\rightarrow j_3,\dots, i_r\rightarrow j_r$. If $r=3$, then 
consider the parameters $x=i_3, y= \alpha, k_1=i_1, k_2=i_2, k_3=\beta \notin \{i_1,i_2,i_3,k,\alpha\}$.
Otherwise let $x=i_4,y= \alpha, k_1=i_1, k_2=i_2, k_3=i_3$. Generate $\sigma_2,\dots,\sigma_6,
\sigma'_1,\dots,\sigma'_6$. Corresponding vertices either belong to $S$ or to
$\cup_{2\leq x<r}T_{3,x}$.
 So using induction on $r$ and the result of lemma \ref{3-2} as the base case, 
lemma \ref{zero} gives that $P^{[2]}_{\sigma_1}\in LS(T_1\cup S)$.

Similar argument shows that $T_{2,r}$ vertices also belong to $LS(T_1\cup S)$.
$\Box$
\end{proof}

\begin{corollary}\label{T23}
If $n\geq 6$, then $T_2\cup T_3\subset LS(T_1\cup S)$.
\end{corollary}

\proof{
Follows from the previous three lemmas. $\Box$
}

\begin{theorem}\label{expset}
If $n\geq 6$, then inequality (\ref{ineqsys}) defines a facet of $\mathcal{B}^{[2]}$.
\end{theorem}

\proof{
From corollaries \ref{T1} and \ref{T23} every vertex in $T$ can be expressed as a linear combination of a 
fixed vertex in $T$ and the vertices of $S$. Now the result follows from lemma \ref{st}. $\Box$
}

Together $\frac{n^2(n-1)^2}{2}+\sum_{i=2}^{n-3}$ $\frac{n^2(n-1)^2\cdots(n-i)^2}{i!}$ facets of the polytope are
defined by theorems \ref{nonneg},\ref{secondset},\ref{expset}.

\subsection{Previously Known Facets of ${\cal B}^{[2]}$}

Let $P_1,P_2$ be disjoint subsets of $[n]$. Similarly $Q_1,Q_2$ are also disjoint subsets of $[n]$. 
Then the $4$-box inequality discussed in \cite{JungerK97,KaibelPhD97} is 
$(-\sum_{i\in P_1,j\in Q_1}u_{ij} - \sum_{i\in P_2,j\in Q_2}u_{ij}$ 
$ + \sum_{i\in P_1,j\in Q_2}u_{ij} + \sum_{i\in P_2,j\in Q_1}u_{ij}$ 
$ - (\beta-0.5)\omega)^2 \geq 0.25$. The $1$-box inequality is equivalent to $(\sum_{i\in P_1,j\in Q_2}u_{ij}-$ 
$(\beta-0.5)\omega)^2 \geq 0.25$ and is obtained by setting $P_2=Q_1=\emptyset$ in the $4$-box inequality, whereas
the $2$-box inequality corresponds to $(- \sum_{i\in P_2,j\in Q_2}u_{ij}$ $ + \sum_{i\in P_1,j\in Q_2}u_{ij}$ 
$- (\beta-0.5)\omega)^2 \geq 0.25$ and is obtained by setting $Q_1=\emptyset$ in the $4$-box inequality.
All the facets listed in \cite{JungerK97,KaibelPhD97} are special instances of either the $1$-box or the
$2$-box inequality.

\section{Insufficiency of inequality (\ref{gen-ineq})} \label{apx_insuf}

We will show that even after including every facet given by the inequality
(\ref{gen-ineq})
in the SDP relaxation of CP, the resulting feasible region remains larger than ${\cal B}^{[2]}$.
Hence there exist facets of ${\cal B}^{[2]}$ which are yet to be discovered.
In the following we will replace the inequality (\ref{gen-ineq}) with the equivalent 
$\sum_{ijkl} x_{ij}x_{kl}$ $Y_{ij,kl}$ $- (2z-1) \sum_{ij}x_{ij}Y_{ij,ij}$ $+ z^2 - z\geq 0 $.
Further let ${\cal F}$ denote the feasible region of SDP where the objective function attains its maximum
value, $n$.

\begin{lemma}\label{ppp} Let $P_{\sigma}$ denote the row-major vectorization of the
corresponding permutation matrix. Following statements are equivalent.

1. Region of ${\cal F}$, satisfying conditions
$\sum_{ijkl} x_{ij}x_{kl}$ $Y_{ij,kl}$ $ - (2z-1) \sum_{ij}x_{ij}Y_{ij,ij}$ $ + z^2 - z\geq 0 $
for all $x_{ij}, z\in \mathbb{Z}$ is exactly equal to ${\cal B}^{[2]}$.

2. Given any set of permutations $I$ such that $\{P^{[2]}_{\sigma}|\sigma\in I\}$ is L.I.
Then $\sum_{\sigma\in I}\alpha_{\sigma}((P_{\sigma}^T\cdot x)^2 -(2z-1) (P_{\sigma}^T\cdot x)) +z^2
- z\geq 0$ for all $x\in \mathbb{Z}^{n^2}, z\in \mathbb{Z}$
if and only if $\alpha_{\sigma}\geq 0\ \forall\ \sigma\in I$ and $\sum_{\sigma\in I}\alpha_{\sigma}=1$.
\end{lemma}

\begin{proof}
Let $Y$ be a point in the feasible region of the SDP. We know from lemma \ref{affine} that
 $Y$ is in the affine hull of $P^{[2]}_{\sigma}$s. So there exist $\alpha_{\sigma}$ such
that $Y=\sum_{\sigma}\alpha_{\sigma}P^{[2]}_{\sigma}$ where $\sum_{\sigma} \alpha_{\sigma}= 1$ and
$\{P^{[2]}_{\sigma}|\alpha_{\sigma}\neq 0\}$ is linearly independent.

Consider arbitrary $x \in \mathbb{Z}^{n^2}$ and $z\in \mathbb{Z}$. So
$\sum_{ijkl} x_{ij}x_{kl}$ $Y_{ij,kl}$ $ - (2z-1) \sum_{ij}x_{ij}Y_{ij,ij}$ $ +z^2 - z$
$= \sum_{\sigma}\alpha_{\sigma}\sum_{ijkl}x_{ij}x_{kl}$ $ P^{[2]}_{\sigma}(ij,kl)$
$- (2z-1)\sum_{\sigma} \alpha_{\sigma}\sum_{ij}x_{ij}$ $P^{[2]}_{\sigma}(ij,ij)$ $ +z^2 - z$
$= \sum_{\sigma}\alpha_{\sigma}\sum_{ijkl}x_{ij}x_{kl}(P_{\sigma})_{ij}(P_{\sigma})_{kl}$
$- (2z-1)\sum_{\sigma} \alpha_{\sigma}\sum_{ij}x_{ij}$ $(P_{\sigma})_{ij}$ $ +z^2 - z$
$=  \sum_{\sigma}\alpha_{\sigma}(P^T_{\sigma}\cdot x)^2$ $ - (2z-1)\sum_{\sigma}\alpha_{\sigma}
(P^T_{\sigma}\cdot x)$ $ +z^2 - z$.

Besides, $Y\in{\cal B}^{[2]}$ if and only if $\alpha_{\sigma}\geq 0\ \forall\ \sigma$. $\Box$
\end{proof}

We first prove a useful lemma. In the following let $P^T_{\sigma}\cdot x +z =\sum_{i=1}^n
x_{i,\sigma(i)}+z$ be denoted by $y_{\sigma}$.

\begin{lemma}\label{sss} $\sum_{\sigma}\alpha_{\sigma}(y^2_{\sigma}-y_{\sigma})=0$ for all
$x\in \mathbb{Z}^{n^2},$ $z\in \mathbb{Z}$ if and only if $\sum_{\sigma}\alpha_{\sigma}y^2_{\sigma}=0$ 
for all $x\in \mathbb{Z}^{n^2},$ $z\in\mathbb{Z}$.
\end{lemma}

\begin{proof} (If) Let $S(x,z) = \sum_{\sigma}\alpha_{\sigma}y^2_{\sigma}$. We have
$S(x,z)=0$ for all $x\in \mathbb{Z}^{n^2}$ and $z\in \mathbb{Z}$.
Let $a$ be any arbitrary point in $\mathbb{Z}^{n^2}$. Fix some $i,j$. Define $a'$ as
$a'_{i'j'}=a_{i'j'}$ if $i'\neq i$ or $j'\neq j$ and $a'_{ij}=a_{ij}+1$. So $S(a',b)=S(a,b)
+\sum_{\sigma:\sigma(i)=j}\alpha_{\sigma} (2y_{\sigma}(a,b) +1)$.
Define $a''$ in the similar way as $a'$ is defined, except here $a''_{ij} = a_{ij}-1$. Then
we get $S(a'',b)= S(a,b) -\sum_{\sigma:\sigma(i)=j}\alpha_{\sigma} (2y_{\sigma}(a,b) -1)$.
So $(S(a',b)-S(a'',b))/2 = \sum_{\sigma:\sigma(i)=j}\alpha_{\sigma} y_{\sigma}(a,b)$. Setting
$S(a',b)=S(a'',b)=0$ we have $\sum_{\sigma:\sigma(i)=j}\alpha_{\sigma}
y_{\sigma}(a,b)=0$. So $\sum_{\sigma}\alpha_{\sigma}y_{\sigma}(a,b)
= \sum_j\sum_{\sigma:\sigma(i)=j}\alpha_{\sigma}y_{\sigma}(a,b)=0$. As $a$ is arbitrarily
chosen we have $\sum_{\sigma}\alpha_{\sigma}y_{\sigma}=0$ for all $x\in \mathbb{Z}^{n^2}$
and all $z\in \mathbb{Z}$.

(Only if) This part is trivial because $S(x,z)=0.5($ $T(x,z)$ $+T(-x,-z))$ where
$T(x,z)=\sum_{\sigma}\alpha_{\sigma}$ $(y^2_{\sigma}-y_{\sigma})$.
$\Box$
\end{proof}

Let $p_{\sigma}$ be the $(n^2+1)$-dimensional vector in which the first $n^2$ entries are the vectorized $P_{\sigma}$
and the last entry is $1$. Define $\tilde{P}^{[2]}_{\sigma}=p_{\sigma}.p_{\sigma}^T$.

\begin{lemma} $\{{P}^{[2]}_{\sigma}|\sigma\in I\}$
is linearly independent if and only if $\{y_{\sigma}^2-y_{\sigma}|\sigma\in I\}$ is linearly independent.
\end{lemma}

\begin{proof} The $n^2\times n^2$ matrix $P^{[2]}_{\sigma}$ is a principal submatrix of
$\tilde{P}^{[2]}_{\sigma}$. The last row and the last column of $\tilde{P}^{[2]}_{\sigma}$
is the same as its diagonal. Hence $\{P^{[2]}_{\sigma}|\sigma\in I\}$
is linearly independent if and only if $\{\tilde{P}^{[2]}_{\sigma}|\sigma\in I\}$
is linearly independent.

$\sum_{\sigma\in I}\alpha_{\sigma}\tilde{P}^{[2]}_{\sigma} = 0$
if and only if
$\sum_{\sigma\in I}\alpha_{\sigma}q^T \tilde{P}^{[2]}_{\sigma}q = 0\; \forall q\in \mathbb{Q}^{n^2+1}$,
where first $n^2$ components of $q$ is $x$ and the last component is $z$. It can be rewritten as
$\sum_{\sigma \in I}\alpha_{\sigma}(P^T_{\sigma}\cdot x+z)^2 = 0\; \forall x\in \mathbb{Q}^{n^2},\, \forall z\in \mathbb{Q}$.
Writing in terms of $y_{\sigma}$, the above statement is equivalent to
$\sum_{\sigma \in I}\alpha_{\sigma}y_{\sigma}^2 = 0 \; \forall x\in \mathbb{Q}^{n^2}
\, \forall z\in \mathbb{Q}$.
From lemma \ref{sss}, this is equivalent to $\sum_{\sigma\in I}\alpha_{\sigma}(y_{\sigma}^2-y_{\sigma}) = 0 \;
\forall x\in \mathbb{Q}^{n^2},\, \forall z\in \mathbb{Q}$. $\Box$
\end{proof}

Consider the polynomial ring $A=\mathbb{Q}[\{x_{ij}|1\leq i,j\leq n\}\cup \{z\}]$.
 The subspace of $A$ generated by $\{x_{ij}|1\leq i,j\leq n\} \cup \{z\} \cup
\{x_{ij}x_{kl}|1\leq i,j,k,l\leq n\} \cup \{zx_{ij}|1 \leq i,j\leq n\}$
is the direct sum of components of degree $1$ and $2$. Its dimension is $d=1+(n^2+1)+(n^4+n^2)/2$. For $n\geq 6$,
$d\leq n!$. So the set $\{y^2_{\sigma}-y_{\sigma}| \sigma\in S_n\}$ is linearly dependent for all $n\geq 6$.

Let $J$ be a minimal set of permutations such that $\{y^2_{\sigma}-y_{\sigma}|\sigma \in J\}$ is linearly dependent.
So there exist $\alpha_{\sigma}$ such that $\sum_{\sigma\in J}\alpha_{\sigma}(y^2_{\sigma}-y_{\sigma})=0$. Since no
set of two $P^{[2]}_{\sigma}$ is L.D., the same holds for any pair of $y^2_{\sigma}-y_{\sigma}$. Hence at least
three coefficients are non-zero. Assume that $\alpha_{\sigma_1},\alpha_{\sigma_2},\alpha_{\sigma_3}$
are non-zero. Let the sign of the first two be same. We may assume that $\alpha_{\sigma_1}$ and $\alpha_{\sigma_2}$ are
negative. If not, then invert the sign of every coefficient. Note that $(-\alpha_{\sigma_1})(y^2_{\sigma_1}-y_{\sigma_1})$
is non-negative for all $x\in \mathbb{Z}^{n^2}, z\in \mathbb{Z}$. So $\sum_{\sigma\in J}\alpha_{\sigma}
(y^2_{\sigma}-y_{\sigma}) + (-\alpha_{\sigma_1})(y^2_{\sigma_1}-y_{\sigma_1})$ is non-negative for all
$x\in \mathbb{Z}^{n^2}, z\in \mathbb{Z}$. This
simplifies to $\sum_{\sigma\in J\setminus \{\sigma_1\}}\alpha_{\sigma}(y^2_{\sigma}-y_{\sigma})$ which is non-negative
for all $x\in\mathbb{Z}^{n^2},z\in \mathbb{Z}$ and $\{y^2_{\sigma}-y_{\sigma}|\sigma \in J\setminus \{\sigma_1\}\}$
is L.I. But $\alpha_{\sigma_2}$ is negative. Hence we have established that the second statement of lemma \ref{ppp}
does not hold.

\begin{theorem} Region of ${\cal F}$ satisfying conditions (\ref{gen-ineq}), properly contains ${\cal B}^{[2]}$. 
\end{theorem}

\begin{corollary} There exists at least one facet of ${\cal B}^{[2]}$ which is not an instance of (\ref{gen-ineq}).
\end{corollary}

\bibliographystyle{plain}
\bibliography{qap_facets}

%\newpage

\appendix
\section{Proof of Lemma \ref{cpfeas}}\label{CPFeasPrf}

\begin{proof}
%$P^{[2]}_{\sigma}$ is a completely positive rank-$1$ matrix because it is the outer product of the vectorized
%$P_{\sigma}$ with itself. If $\sigma$ is an isomorphism
%between $G_1$ and $G_2$ then $P^{[2]}_{\sigma}$ is feasible for CP-LT since it satisfies all the linear
%conditions. Conversely, if $P^{[2]}_{\sigma}$, for some $\sigma$,
%is feasible for CP-LT, then it implies that the product graph $G$ has a clique of size
%$n$ (i.e., the program returns the optimal value $n$) and hence $G_1$ and $G_2$ are isomorphic 
%with $\sigma$ as the isomorphism.

%From the above, ${\cal P}_{G_1G_2}$ is contained in the feasible region of CP-LT with objective function
%equal to $n$.
Consider a non-negative vector realization $\{u_{ij}|$ $i,j\in[n]\}\cup \{\omega\}$
for a point $Y\in{\cal L}$. Let $W$ denote an $n\times n$ matrix with $(i,j)$-th entry being $u_{ij}$.
Conditions \ref{1.2} and \ref{1.3} ensure that vectors in any row or any column of $W$ are pairwise orthogonal.
Since objective function attains value $n$, from Observation \ref{obs1} vectors of each row/column 
form a maximal set of pairwise orthogonal united vectors. Also from the same observation each row 
and each column adds up to $\omega$. Assume that the vector realization is in an $N$-dimensional space.
 Consider the $r$-th component of the matrix, i.e., the matrix formed by the $r$-th component of each vector. 
Let us denote it by $D_r$. Each element of $D_r$ is non-negative and each row and each
column adds up to $\omega_r$, the $r$-th component of $\omega$. Hence $D_r$ is $\omega_r$ times a
doubly-stochastic matrix. But the vectors of the same row (resp. column) are orthogonal so exactly one entry 
is non-zero in each row (resp. column) if $\omega_r >0$. So $D_r = \omega_rP_{\sigma_r}$ for some permutation $\sigma_r$.
We can express $W$ by $\sum_r \omega_rP_{\sigma_r}e_r$ where $e_r$ denotes the unit vector along the
$r$-th axis. $Y_{ij,kl}$ is the inner product of the vectors $u_{ij}$ and $u_{kl}$
which is $(\sum_r \omega_r(P_{\sigma_r})_{ij}e_r)\cdot (\sum_s \omega_s(P_{\sigma_s})_{kl}e_s)
= \sum_r \omega_r^2 (P_{\sigma_r})_{ij}.(P_{\sigma_r})_{kl} =
\sum_r \omega_r^2(P^{[2]}_{\sigma_r})_{ij,kl}$. Thus $Y =$ $\sum_r$ $\omega_r^2P^{[2]}_{\sigma_r}$.
Since $\sum_r\omega_r^2 = \omega^2 =1$, $Y$ is a convex combination of some of the $P^{[2]}_{\sigma}$s.
%Each $\sigma_r$, where $\omega_r >0$, is an isomorphism between $G_1$ and $G_2$. Hence $Y\in{\cal P}_{G_1G_2}$.
%So the feasible region is contained in ${\cal P}_{G_1G_2}$.
$\Box$ 
\end{proof}

\section{Proof of Lemma \ref{p2s-01}}\label{apx_p2s-01}

\begin{proof} Let $Y$ be a $0/1$ solution of the system of linear equations \ref{2.1}-\ref{2.4}.
Note that equations \ref{2.4} and the non-negativity of the entries ensure that the diagonal of
the solution is a vectorized doubly stochastic matrix. As the solution is a $0/1$ matrix, the diagonal must
be a vectorized permutation matrix, say $P_{\sigma}$. Then $Y_{ij,ij}=(P_{\sigma})_{ij}$.

Equations \ref{2.3} imply that $Y_{ij,kl}=1$ if and only if $Y_{ij,ij}=1$
and $Y_{kl,kl}=1$. Equivalently, $Y_{ij,kl}= Y_{ij,ij}. Y_{kl,kl} =
(P_{\sigma})_{ij}. (P_{\sigma})_{kl} = (P^{[2]}_{\sigma})_{ij,kl}$.

Equations \ref{2.1} and \ref{2.2} describe the remaining entries. $\Box$
\end{proof} 

\section{Proof of Lemma \ref{affine}} \label{affine-plane}

\begin{proof} We will first show that the dimension of the solution plane is no more than
$n!/(2(n-4)!) +(n-1)^2 +1$.

Split matrix $Y$ into $n^2$ non-overlapping sub-matrices of size $n\times n$
which will be called {\em blocks}. The $n$ blocks that contain the diagonal entries of $Y$ will be called
{\em diagonal blocks}. Note that $Y_{ij,kl}$ is the $jl$-th entry of the $ik$-th block.

From the equation \ref{2.2}, the off-diagonal entries of the diagonal blocks are zero. Assume that
the first $n-1$ diagonal
entries of the first $n-1$ diagonal blocks are given. Then all diagonal entries can be determined using equations
\ref{2.4}.

Consider any off diagonal block in the region above the main diagonal, other than the right most ($n$-th)
block of that row. Note that the first entry of such a block will be $Y_{r1,s1}$ where
$r < s < n$. From the equation \ref{2.2} we see that its diagonal entries are zero. The sum of the entries of
any row of this block is same as the main diagonal entry of that row in $Y$, see equation \ref{2.3}.
Same holds for the columns from symmetry condition \ref{2.1}. Hence by fixing all but one off-diagonal
entries of the first
principal sub-matrix of the block of size $(n-1)\times (n-1)$, we can fill in all the remaining entries.
An exception to above is the second-last block of the $(n-2)$-th block-row (with first entry $Y_{(n-2)1,(n-1)1}$).
Here only the upper diagonal entries of the first principal sub-matrix of size $(n-1)\times (n-1)$ are sufficient
to determine all the remaining entries of that block. From equation \ref{2.3} all the entries of the right most blocks
can be determined. Lower diagonal entries of $Y$ are determined by symmetry.
Hence we see that the number of free variables is no more than $(n-1)^2 + ((n-1)(n-2)-1)(2+\dots+(n-2)) + (n-1)(n-2)/2
= n!/(2(n-4)!) + (n-1)^2 +1$.

In \cite{KaibelPhD97} it is shown that the dimension of ${\cal B}^{[2]}$ polytope is $\frac{n!}{2(n-4)!} + (n-1)^2 +1$.
%We have also given an alternative proof of the same claim in \ref{PrfDim}.
This claim along with the result of the previous paragraph leads to the conclusion that equations \ref{2.1}-\ref{2.4}
define the affine plane spanned by the $P^{[2]}_{\sigma}$s.
$\Box$
\end{proof}

\end{document}